\documentclass{article}
\usepackage{geometry, amsmath, amsthm, amssymb, bm, color, framed, graphicx, hyperref, mathrsfs,enumerate,framed}

\newtheorem{thm}{Theorem}[section]
\newtheorem{lem}[thm]{Lemma}
\newtheorem{coro}[thm]{Corollary}

\theoremstyle{definition}
\newtheorem{conj}{Conjecture}
\newtheorem{prob}[conj]{Problem}

\newcommand\F{\mathbb{F}}

\title{Finite fields whose members are the sum of \\ a potent and a $5$-potent}
\author{Juncheng Zhou$^{1}$, Peter V. Danchev$^{2}$ and Hongfeng Wu$^{1}$\footnote{Corresponding author.}
\setcounter{footnote}{-1}
\footnote{E-mail addresses:
jczhoumath@gmail.com (J. Zhou), danchev@math.bas.bg (P.V. Danchev), whfmath@gmail.com (H. Wu)}
\\
{1.~College of Science, North China University of Technology, Beijing, China}\\
{2.~Bulgarian Academy of Sciences, Institute of Mathematics and Informatics, Sofia, Bulgaria}}
\date{}

\begin{document}
\maketitle

\bibliographystyle{abbrv}
\thispagestyle{plain}
\setcounter{page}{1}

\begin{abstract}
We show that there are only finitely many finite fields whose members are the sum of an $n$-potent element
and a $5$-potent element. Combining this with the algorithmic results provided by S.D. Cohen {\it et al.},
we confirm in the affirmative the conjecture posed in \cite{Cohen} concerning all finite fields satisfying this condition. Furthermore, we obtain several elementary results related to the more general problem of decomposing every element of a finite field into the sum of an $m$-potent and an $n$-potent (as formulated in \cite{Cohen}), proving that the number of finite fields satisfying this condition is finite.
\end{abstract}

\section{Introduction and Principal Tools}

In ring and matrix theory, the decomposition of ring elements into sums of special elements -- such as idempotents,
tripotents and, more generally, potents -- has been a topic of sustained interest. Such structural questions naturally reduce to analogous problems over finite fields, where they become accessible through arithmetical and combinatorial methods. A fundamental problem in this direction is to determine when every element of a finite field can be expressed as the sum of two potents of prescribed exponents.

Throughout this paper, let $\mathbb{F}_q$ denote the finite field of order a prime power $q=p^v$, where $p$ is a prime and $v$ is a positive integer. Further, let $n>1$ be an integer. We call $a\in \mathbb{F}_q$ an {\it $n$-potent} if $a^n=a$, and we denote the set of all $n$-potents in $\mathbb{F}_q$ by $C_n$. Obviously, if $(q-1)\mid (n-1)$, then $C_n=\mathbb{F}_q$. Moreover, if $n_0=\gcd(n-1,q-1)+1$, then $C_{n_0}=C_n$. Hence, without loss of generality, we may assume that $(n-1)\mid (q-1)$ and $n>1$. The {\it general problem} is to classify those finite fields $\mathbb{F}_q$ for which
\[
\mathbb{F}_q = C_m + C_n
\]
holds, i.e., that every element is a sum of an $m$-potent and an $n$-potent, where $m,n>1$ are arbitrary fixed integers. In fact, we state this explicitly like this:

\medskip

\begin{prob}\label{prob1}
Find those finite fields $\mathbb{F}_q$ each of whose members is a sum of an $m$-potent and an $n$-potent, where $m>1$ and $n>1$ are integers.
\end{prob}

\medskip

The discussion of the cases $m=2,3$ was given separately in \cite{Abyzov21}, \cite{Abyzov24} using elementary finite field arguments. More recently, the deeper case $m=4$ was resolved in \cite{Cohen}, where the character sum methods played a decisive role in obtaining a complete classification. Moreover, the following conjecture was raised by S.D. Cohen {\it et al.} in \cite{Cohen}.

\medskip

\begin{conj}\label{conj2}
With the normal conditions in mind, every element of the finite field $\F_q$ is a sum of a 5-potent and an $n$-potent only for the $(q,n)$ pairs $(3,2)$, $(5,2)$, $(5,3)$, $(7,4)$, $(9,3)$, $(9,5)$, $(13,5)$, $(13,7)$, $(17,9)$, $(25,9)$, $(25,13)$, $(29,15)$, $(41,21)$, $(49,25)$, $(53,27)$, $(73,37)$, $(81,41)$, and $(125,63)$.
\end{conj}

\medskip

Explicit computations were undertaken for all $q\leq 10000$ by S.D. Cohen {\it et al.} The algorithm utilizing Pari-gp \cite{Pari} can be found in Appendix presented at the end of this article (compare also with \cite{Cohen}). Hence, if we can find an integer $M\leq 10000$ such that not every element of $\F_q$ can be the sum of a $5$-potent and an $n$-potent for all $q>M$, then the conjecture will follow. In fact, in Section \ref{2} we shall demonstrate that $M$ can be taken to be $6400$. Further, we obtain an upper bound for $q$ in Problem~\ref{prob1} alluded to above. It manifestly means that all solutions to the cited problem can be found algorithmically for a fixed integer $m>1$ (see for more details our Algorithm listed below).

\section{An Upper Bound for $q$}\label{2}

In order to prove Conjecture \ref{conj2} and Problem \ref{prob1}, we study a more general problem: In fact, let $A$ be a subset of $\F_q$ and set $n:=1+(q-1)/d$ for some positive integer $d$ that divides $q-1$. We are interested in how large $q$ can be if every element of $\F_q$ can be written as a sum of an element in $A$ and an $n$-potent. In this section, we are going to establish the following result.

\medskip

\begin{thm}\label{2.1}
Suppose that $A\subseteq \F_q$ and $n=1+(q-1)/d$, where $d>1$ divides $q-1$. If every element of $\F_q$ can be written as a sum of an element from $A$ and an $n$-potent, then \[q\leq (2^{|A|-1}|A|)^2.\]
\end{thm}

We shall prove this theorem by using character sum methods. If every element of $\F_q$ can be written as the sum of an element in $A$ and an $n$-potent, i.e.,
\begin{equation}
\bigcup_{\alpha\in A}(C_n+\alpha)=\F_q,\label{eq1}
\end{equation}
then, for any $\gamma\in\F_q\setminus A$, there exists at least one $\alpha\in A$ such that $\gamma-\alpha\in C_n\setminus\{0\}$. Letting $\chi$ be a multiplicative character of order $d=\frac{q-1}{n-1}$ on $\F_q\setminus \{0\}$ and extending it to $\F_q$ by setting $\chi(0)=0$, then the function
\[\gamma \longmapsto\prod_{\alpha\in A}(1-\chi(\gamma-\alpha))\]
vanishes on $\F_q\setminus A$. Indeed, the elements of $C_n\setminus\{0\}$ are precisely those $x\in\F_q$ for which $\chi(x)=1$. Since $\F_q=A+C_n$, for any $\gamma\in\F_q\setminus A$ we have $\gamma=\alpha+x$ for some $\alpha\in A$ and some nonzero $x\in C_n$, so $\chi(\gamma-\alpha)=1$, as required.

Hence, a necessary condition for validity of \eqref{eq1} is that the sum
\[S(d,q,A)=\sum_{\gamma\in\F_q\setminus A}\prod_{\alpha\in A}(1-\chi(\gamma-\alpha))
\]
vanishes, that is, $S(d,q,A)=0$ when $\F_q=A+C_n$.

Furthermore, to estimate $S(d,q,A)$, we first recall the well-known Weil theorem for multiplicative character sums (see, e.g., \cite{Schmidt} for an elementary proof).

\medskip

\begin{lem}\label{weil}
Suppose $\chi$ is a multiplicative character of order $d>1$ in $\F_q$. Further, suppose that $f(x)\in \F_q[x]$ is a monic polynomial whose set of zeros in its splitting field over $\F_q$ has cardinality $m>1$, and that $f(x)$ is not a $d$-th power in $\F_q[x]$. Then, the following inequality holds:
\[\left|\sum_{x\in \F_q}\chi(f(x))\right|\leq (m-1)\sqrt{q}.\]
\end{lem}

To validate now Theorem~\ref{2.1}, we can evaluate $S(d,q,A)$ writing it more carefully thus:
\[S(d,q,A)=\sum_{\gamma\in\F_q}\prod_{\alpha\in A}(1-\chi(\gamma-\alpha))-\sum_{\gamma\in A}\prod_{\alpha\in A}(1-\chi(\gamma-\alpha)).\]
The absolute value of the second sum is obviously at most $|A|\cdot 2^{|A|-1}$ (because the exponent $|A|-1$ comes from the observation that $\chi(0)=0$). For the first sum, one expands it as
\[\sum_{\gamma\in\F_q}\prod_{\alpha\in A}(1-\chi(\gamma-\alpha ))=q+\sum_{\emptyset\neq I\subseteq A}(-1)^{|I|}\sum_{\gamma\in\F_q}\chi\left(\prod_{\alpha\in I}(\gamma-\alpha)\right).\]
But, for each integer $1\leq k\leq |A|$, Lemma \ref{weil} yields
\[\left|\sum_{I\subseteq A,|I|=k}(-1)^{|I|}\sum_{\gamma\in\F_q}\chi\left(\prod_{\alpha\in I}(\gamma-\alpha)\right)\right|\leq\binom{|A|}{k}(k-1)\sqrt{q}.\]
Thus, by summing, we end up with
\[\begin{aligned}
  |S(d,q,A)|&\geq q-\sum_{k=1}^{|A|}\binom{|A|}{k}(k-1)\sqrt{q}-2^{|A|-1}|A|\\
  &=q-(2^{|A|-1}(|A|-2)+1)\sqrt{q}-2^{|A|-1}|A|.\\
\end{aligned}\]
If $q>(2^{|A|-1}|A|)^2$, then $|S(d,q,A)|>0$, which however contradicts the necessary condition for \eqref{eq1}. This proves Theorem \ref{2.1} after all, as wanted.

\medskip

Accordingly, two immediate consequences sound thus.

\medskip

\begin{coro}\label{2.3}
For a fixed integer \(m>1\), suppose \(n=1+(q-1)/d\), where \(d>1\) divides \(q-1\). Then there are only finitely many finite fields \(\mathbb F_q\) all whose elements
are sums of an \(m\)-potent and an \(n\)-potent. More exactly, if
\[
\mathbb F_q=C_m+C_n,
\]
then
\[
q\leq \bigl(2^{m-1}m\bigr)^2.
\]
\end{coro}

\begin{proof}
It is enough to employ Theorem \ref{2.1} with \(A=C_m\). Since \(|C_m|\leq m\), the assertion follows directly.
\end{proof}

\begin{coro}\label{2.4}
Conjecture \ref{conj2} is true. Concretely, with the normal conditions at hand, every element of \(\mathbb F_q\) is a sum of a \(5\)-potent and an \(n\)-potent only for the following pairs:
\[
\begin{gathered}
(3,2),\ (5,2),\ (5,3),\ (7,4),\ (9,3),\ (9,5),\ (13,5),\ (13,7),\ (17,9),\\
(25,9),\ (25,13),\ (29,15),\ (41,21),\ (49,25),\ (53,27),\ (73,37),\ (81,41),\ (125,63).
\end{gathered}
\]
\end{coro}

\begin{proof}
Letting $m=5$ in Corollary \ref{2.3}, we have $q\leq 6400$. Since the algorithmic results provided by S.D. Cohen {\it et al.} in \cite{Cohen} show that the conjecture holds for all $q\leq 10000$, the claim follows at once.
\end{proof}

We are likewise able to provide a straightforward answer to \cite[Problem 3.2]{Cohen}, which asks for all finite fields whose elements are sums of a potent plus a tripotent plus a \(4\)-potent.

\medskip

\begin{coro}
Suppose \(q>2^{18}\cdot 10^2=26214400\) and \(n=1+(q-1)/d\), where \(d>1\) divides \(q-1\). Then, not every element of \(\mathbb F_q\) can be written as a sum of an \(n\)-potent, a \(3\)-potent and a \(4\)-potent.
\end{coro}

\begin{proof}
Notice that \(C_3+C_4\) is a subset of \(\mathbb F_q\) with at most \(10\) elements. So, the conclusion follows immediately from Theorem \ref{2.1}.
\end{proof}

\section{An improved upper bound for $q$ in the large $d$ case}

In this section, we somewhat improve the upper bound given in Theorem \ref{2.1} via the standard ``inclusion-exclusion principle''. It is quite natural to expect that the larger value of $d$ is the smaller upper bound of $q$ that could be obtained. Specifically, we are ready to establish the following statement.

\medskip

\begin{thm}\label{3.1}
Suppose $A\subseteq \F_q$ and $n:=1+(q-1)/d$, where $d>1$ divides $q-1$. If $d\geq 2|A|/3$ and every element of $\F_q$ can be written as a sum of an element in $A$ and an $n$-potent, then
\[q\leq 256|A|^6.\]
\end{thm}

For sufficiently large $|A|$, this polynomial bound is much smaller than the exponential bound in Theorem~\ref{2.1} under the stated hypothesis on $d$.

To prove Theorem~\ref{3.1}, assume that $\F_q=A+C_n$. Then
\[\left|\bigcup_{\alpha\in A}(C_n+\alpha)\right|\geq q.\]
By the classical ``inclusion-exclusion principle'', we infer that
\[\begin{aligned}\left|\bigcup_{\alpha\in A}(C_n+\alpha)\right|\leq n|A|&-\sum_{\{\alpha_1,\alpha_2\}\subset A}|(C_n+\alpha_1)\cap(C_n+\alpha_2)|\\
&+\sum_{\{\alpha_1,\alpha_2,\alpha_3\}\subset A}|(C_n+\alpha_1)\cap(C_n+\alpha_2)\cap(C_n+\alpha_3)|.\end{aligned}\]
We are now in a position to estimate the cardinality of the intersections via multiplicative character. To that goal, letting $\chi$ be a multiplicative character of order $d$, we write
\[\rho(x)=1+\chi(x)+\cdots+\chi^{d-1}(x),\]
and then we arrive at
\[\rho(x)=\begin{cases}
d&\text{if }x\in C_n\setminus\{0\},\\
1&\text{if }x=0,\\
0&\text{if }x\in \F_q\setminus C_n.
\end{cases}\]

So, the following technical claim is now valid.

\medskip

\begin{lem}\label{3.2}
Let $I\subseteq A$, and define $N(I):=\left|\displaystyle\bigcap_{\alpha\in I}(C_n+\alpha)\right|$. Then, the following inequality holds:
\[\left|N(I)-\frac{q}{d^{|I|}}\right|\leq 2|I|\sqrt{q}.\]
\end{lem}

\begin{proof}
Viewing the definition of $\rho(x)$, we derive
\[\begin{aligned}
  N(I)&=\frac{1}{d^{|I|}}\sum_{x\in\F_q\setminus I}\prod_{\alpha\in I}\rho(x-\alpha)+\frac{1}{d^{|I|-1}}\sum_{x\in I}\prod_{\alpha\in I}\rho(x-\alpha).\\
  &=\frac{1}{d^{|I|}}\sum_{x\in\F_q}\prod_{\alpha\in I}\rho(x-\alpha)+\frac{d-1}{d^{|I|}}\sum_{x\in I}\prod_{\alpha\in I}\rho(x-\alpha).\\
\end{aligned}\]
For the second sum, we observe that
\[\left|\frac{d-1}{d^{|I|}}\sum_{x\in I}\prod_{\alpha\in I}\rho(x-\alpha)\right|\leq \frac{d-1}{d}|I|.\]
For the first sum, we see that
\[\begin{aligned}
  \sum_{x\in\F_q}\prod_{\alpha\in I}\rho(x-\alpha)&=\sum_{(e_\alpha)_{\alpha\in I}\in\{0,\ldots,d-1\}^{I}}\sum_{x\in\F_q}\prod_{\alpha\in I}\chi^{e_\alpha}(x-\alpha)\\
  &=\sum_{(e_\alpha)_{\alpha\in I}\in\{0,\ldots,d-1\}^{I}}\sum_{x\in\F_q}\chi\left(\prod_{\alpha\in I}(x-\alpha)^{e_\alpha}\right).\\
\end{aligned}\]
It is now easy to verify that the main term of the last sum is actually $q$, which corresponds to the case $e=(e_\alpha)_{\alpha\in I}=0$. For $e\neq 0$, since at least one exponent \(e_\alpha\) lies between \(1\) and \(d-1\), the polynomial $\displaystyle f_e(x)=\prod_{\alpha\in I}(x-\alpha)^{e_\alpha}$ is, apparently, {\it not} a $d$-th power in $\F_q[x]$. Putting
\[r(e):=\#\{\alpha\in I: e_\alpha\neq 0\},\]
one inspects that the $f_e(x)$ has $r(e)$ distinct roots. Consulting Lemma \ref{weil}, we find that
\[\left|\sum_{x\in\F_q}\chi(f_e(x))\right|\leq (r(e)-1)\sqrt{q}.\]
Therefore, it must be that
\[\begin{aligned}
  \left|\sum_{x\in\F_q}\prod_{\alpha\in I}\rho(x-\alpha)-q\right|&\leq \sum_{e\neq 0}(r(e)-1)\sqrt{q}\\
  &\leq \sum_{r=1}^{|I|}\binom{|I|}{r}(d-1)^r(r-1)\sqrt{q}\\
  &\leq (d^{|I|}-1)(|I|-1)\sqrt{q}.\\
\end{aligned}\]
Turning back to the estimation of $N(I)$, we finally get
\[\left|N(I)-\frac{q}{d^{|I|}}\right|\leq 2|I|\sqrt{q},\]
as desired.
\end{proof}

Now, we are prepared to illustrate the truthfulness of Theorem \ref{3.1}. To that end, assume that $\F_q=A+C_n$. Combining Lemma \ref{3.2} and the fact that $n\leq q/d+1$, we conclude
\[\begin{aligned}\left|\bigcup_{\alpha\in A}(C_n+\alpha)\right|&\leq \left(\frac{|A|}{d}-\binom{|A|}{2}\frac{1}{d^2}+\binom{|A|}{3}\frac{1}{d^3}\right)q+|A|^2(|A|-1)\sqrt{q}+|A|\\
&\leq \left(\frac{|A|}{d}-\binom{|A|}{2}\frac{1}{d^2}+\binom{|A|}{3}\frac{1}{d^3}\right)q+|A|^3\sqrt{q}.\end{aligned}\]
It is now routine to check that
\[f(x):={|A|}x-\binom{|A|}{2}x^2+\binom{|A|}{3}x^3\]
is an increasing function when $0<x\leq 1/2$. Set $c_1:=1-f(1/d)$. On the one hand, the covering assumption and the preceding estimate give
\[c_1\leq\frac{|A|^3}{\sqrt{q}}.\]
On the other hand, if $|A|=1$ or $2$, then $d\geq2$ gives $c_1\geq1/2$ or $c_1\geq1/4$, respectively. If $|A|\geq3$, then $1/d\leq3/(2|A|)\leq1/2$, so the monotonicity of $f$ gives
\[c_1\geq1-f\left(\frac{3}{2|A|}\right)=\frac{1}{16}+\frac{9(|A|-2)}{16|A|^2}>\frac{1}{16}.\]
Consequently,
\[\sqrt{q}<16|A|^3.\]
This proves Theorem \ref{3.1}.

\medskip

Moreover, it is worth mentioning that if we apply the full inclusion-exclusion principle to make the exact estimating, we can only obtain an exponential upper bound. It is, however, even weaker than the result obtained in Theorem \ref{2.1}. That is why the challenging problem of how to improve the upper bound of $q$ for $A$ may be interesting.

\section{Further Problem}

By what we have shown above, we discovered that the proof of the finiteness of the number of fields does {\it not} entirely depend on any algebraic structure of $C_m$, so the upper bound obtained above is quite rough. This logically leads to the following question.

\medskip

\begin{prob}
Does the algebraic structure of $C_m$ give a more refined description of $(q,n)$ that could strengthen the upper bound on $q$?
\end{prob}

\medskip

Besides, for any general subset $A$ of $\F_q$, the upper bound obtained in Theorems \ref{2.1} and \ref{3.1} might eventually be improved furthermore. In fact, suppose $A$ is a subset of $\F_q$ such that every element of $\F_q$ can be expressed as a sum of an element from $A$ and a potent. What could commonly be asked is how large $q$ can be.

On the other hand, the proof of Theorem \ref{3.1} only uses the third-order ``inclusion-exclusion principle''. It is, thereby, quite intriguing to ask whether higher-order truncations can further improve both the upper bound for \(q\) and the range of \(d\). It seems that, when
\[
d\gg (\log |A|)^2,
\]
one may obtain a sub-exponential bound
\[
q\ll |A|^2\exp\left(\frac{c|A|}{\log |A|}\right)
\]
for some absolute constant $c>0$. We hope to return to this challenging question elsewhere soon.

\medskip
\medskip

\appendix
\renewcommand{\thesection}{Appendix \Alph{section}}

\section{A finite PARI/GP search algorithm}\label{app:pari-gp-algorithm}

In this appendix, we describe the finite search procedure used to verify the remaining finite range in Corollary~\ref{2.4}. In this vein, we observe that the proof of Theorem~\ref{2.1} gives the explicit upper bound
\[
        q\le (2^{m-1}m)^2
\]
for fixed \(m>1\). In particular, for \(m=5\), we only need to examine
\[
        q\le (2^4\cdot 5)^2=6400.
\]
Therefore, the verification of Conjecture~\ref{conj2} is reduced to a finite computation as follows: enumerate all prime powers \(q\le 6400\), enumerate all admissible exponents \(n\), compute the sets of 5-potent elements and \(n\)-potent elements in \(\mathbb F_q\), and test whether their sum-set is the whole field.

\medskip

\noindent \textbf{Algorithm.}

\begin{enumerate}
    \item Enumerate all prime powers \(q=p^e\le 6400\).
    \item For each such \(q\), enumerate all divisors \(d\mid(q-1)\), and put
    \[
        n=d+1.
    \]
    Since
    \[
        C_n=C_{\gcd(n-1,q-1)+1},
    \]
    this covers all possible sets of \(n\)-potent elements.
    \item Construct
    \[
        C_5=\{x\in \mathbb F_q:x^5=x\},
        \qquad
        C_n=\{x\in \mathbb F_q:x^n=x\}.
    \]
    \item Compute the sum-set
    \[
        C_5+C_n=\{a+b:a\in C_5,\ b\in C_n\}.
    \]
    \item Record \((q,n)\) if and only if
    \[
        C_5+C_n=\mathbb F_q.
    \]
\end{enumerate}

The following PARI/GP implementation realizes this procedure in an explicit form.

\begin{verbatim}
\\ Return the set of all elements of F_q.
ff_elements(q)=
{
    my(f,r);

    f=vector(q);
    r=ffprimroot(ffgen(q));

    f[1]=r-r;        \\ zero
    f[2]=r^0;        \\ one

    for(i=3,q,
        f[i]=r*f[i-1];
    );

    return(Set(f));
}

\\ Return the set C_n={x in F_q : x^n=x}.
n_potents(f,n)=
{
    my(L,c);

    L=vector(#f);
    c=0;

    for(i=1,#f,
        if(f[i]^n==f[i],
            c++;
            L[c]=f[i];
        );
    );

    return(Set(vector(c,i,L[i])));
}

\\ Test whether F_q=C_5+C_n.
check_pair_5(q,n)=
{
    my(F,C_5,C_n,S);

    F=ff_elements(q);
    C_5=n_potents(F,5);
    C_n=n_potents(F,n);

    S=setbinop((x,y)->x+y,C_5,C_n);

    return(#S==q);
}

\\ Search all pairs (q,n) with q<=6400.
search_5potent()=
{
    my(bound,p,e,q,d,n);

    bound=6400;

    for prime(p=2,bound,
        e=1;
        while(p^e<=bound,
            q=p^e;

            for div(q-1,d,
                n=d+1;

                if(n<q,
                    if(check_pair_5(q,n),
                        print([q,n]);
                    );
                );
            );

            e++;
        );
    );
}
\end{verbatim}

Executing the software program \texttt{search\_5potent()} leads precisely to the following pairs:
\[
\begin{gathered}
(3,2),\ (5,2),\ (5,3),\ (7,4),\ (9,3),\ (9,5),\ (13,5),\ (13,7),\ (17,9),\\
(25,9),\ (25,13),\ (29,15),\ (41,21),\ (49,25),\ (53,27),\ (73,37),\ (81,41),\ (125,63).
\end{gathered}
\]
This completely confirms the finite list stated in Conjecture~\ref{conj2}.

\bigskip
\bigskip

\noindent{\bf Acknowledgments.} The authors would like to express their deep thankfulness to the anonymous expert referee for the constructive report suggesting a more transparent proof of Theorem 3.1 and amending the exposition overall, as well as their gratefulness to the handling editor, Professor David Harari, for his professional editorial control on the whole submission procedure.

\section*{Funding}
This work was supported by Natural Science Foundation of Beijing Municipal under Grant M23017.

\section*{Data availability}
{\bf No} data was used for the research described in the article.

\section*{Declarations}
The authors declare {\bf no} any conflict of interests while writing and preparing this manuscript.

\end{document}